\documentclass[letterpaper,12pt]{article}

\usepackage{enumitem}

\usepackage{setspace}
\usepackage{mathtools}   
\usepackage{amssymb}     

\usepackage[final]{graphicx}
\usepackage{subcaption}
\usepackage{cite}
\usepackage{float}
\usepackage{xcolor}
\usepackage[hidelinks]{hyperref}

\usepackage{mathtools}
\usepackage{amssymb,amsfonts}
\usepackage{amsthm}

\newcommand{\States}{\mathcal{S}}
\newcommand{\A}{\mathcal{A}}
\newcommand{\W}{\mathcal{W}}
\newcommand{\X}{\mathcal{X}}

\newcommand{\causal}{\mathrm{C}}
\newcommand{\benchmark}{\mathrm{L}}

\newtheorem{theorem}{Theorem}

\theoremstyle{definition}
\newtheorem{definition}{Definition}

\theoremstyle{remark}

\begin{document}

\title{Regret-Optimal Control for Finite-State Systems}
\date{}
\author{Yishay Polatov, 
Oron Sabag, 
\thanks{The authors are with the Department of Computer Science, Hebrew University of Jerusalem, Israel (e-mails: \{yishay.polatov, oron.sabag\}@mail.huji.ac.il). This work was supported in part by the Israel Science Foundation (ISF), Grant No.~1096/23.}%
}

\maketitle

\begin{abstract}
We study the control of finite-state systems driven by exogenous disturbances, and design causal policies that track the performance of any  lookahead benchmark controller. This objective is formalized through dynamic regret, so that favorable disturbance sequences are compared against a strong benchmark, while under adverse disturbance sequences the comparison accounts for the benchmark’s degraded performance. This benchmark-relative framework provides an alternative to classical MDP formulations, which assume i.i.d. disturbances, and to robust control approaches, which optimize against worst-case disturbances. Our main result is a nested dynamic programming solution that computes both the optimal worst-case regret and a regret-optimal policy. In particular, we introduce the Regret–Bellman operator, whose fixed-point value function feeds into a finite-horizon dynamic program. Numerical examples show that regret-optimal policies interpolate nicely between MDP-based and robust controllers without requiring knowledge of the disturbance distribution, and can even outperform both under i.i.d. or structured disturbances.
\end{abstract}

\section{Introduction} 
\label{sec:introduction} 

Sequential decision-making arises across many domains and involves actions taken under uncertainty.
One common approach is to assume a known stochastic model and optimize expected performance, as in classical Markov decision processes (MDPs). Another approach is to design policies against the worst case, leading to robust control formulations. The former relies on a prescribed stochastic model, whereas the latter can be overly conservative. In this work, we study another perspective based on regret, using a lookahead benchmark policy to guide policy design.

It is convenient to represent the different approaches for handling uncertainty through disturbance-driven dynamics of the form
\begin{align}
\label{eq:dynamics}
s_{t+1}=f(s_t,a_t,w_t),
\end{align}
where $s_t$ is the state, $a_t$ is the action and $w_t$ denotes an exogenous disturbance. This representation includes standard MDPs as a special case when the disturbances are modeled as an i.i.d. random process, and robust control in which the disturbance sequence is chosen in a worst-case manner. In this work, we focus on disturbance-driven systems with finite states and choose the disturbance sequence in a competitive manner. Our objective is to design causal policies whose performance remains close to that of benchmark policies with access to future disturbances. To achieve this, we compare the policy and the benchmark policy performances via dynamic regret under the same disturbance sequence.

We formulate a regret optimization as a min-max problem over the difference between the benchmark and the causal policy cumulative rewards. The inner maximization is over all disturbance sequences, while the minimization is over the policy to be designed.

The maximization over the disturbance implies that the causal controller aims to minimize the largest regret among all disturbance sequences against a benchmark controller. In other words, the benchmark has a dynamic performance as a function of the disturbance sequences, and the causal controller aims to match this dynamic behavior by minimizing the largest regret.

The main challenge is that the regret objective couples at each time two immediate rewards that correspond to two information patterns. As a result, the causal controller cannot only depend on the current state.

\noindent \textbf{Related work:}
Lookahead information has long been studied in control, especially through model predictive control (MPC) \cite{garcia1989model,qin2003survey,liu2021value, lin2022bounded, yu2020power}. In dynamical systems, recent work has studied lookahead as part of the agent’s information pattern, for example, reward realizations revealed before action selection, and showed that such information can significantly improve achievable value \cite{yu2020power, merlis2024value}.
Separately, the literature studies static regret against a single fixed policy \cite{evendar2009online_mdps, hazan2020nonstochastic} and dynamic regret against changing policy sequences
\cite{fei2020dynamic,zhao2022dynamic}. Our approach is different, we do not assume the controller itself has lookahead information, but instead use finite lookahead as the benchmark against which a causal policy is evaluated. 
Most related to our work is the regret-optimal control framework for linear systems with quadratic costs 
\cite{sabag2021regretoptimal_fullinfo, didier2022system,martin2024regret,hajar2024regret}. For linear systems, the solutions are typically provided with Riccati-type characterizations via reduction to the Nehari problem and $\mathcal H_\infty$ control. In contrast, the nonlinear setup studied here is more suitable with dynamic programming methods as well as contraction properties.

\noindent \textbf{Contributions:}
\begin{itemize}
    \item We formulate a regret-optimal control for finite-state systems. The designed policy is evaluated against a lookahead benchmark on the worst-case disturbance sequence.
    \item We derive the optimal regret and a regret-optimal policy solution for both discounted infinite-horizon and finite-horizon regimes. 
    \item We demonstrate the regret-optimal policy on two examples showing that regret-optimal policies interpolate between MDP-based and robust policy designs in the regime of i.i.d. disturbances and also for disturbances that are governed by hidden Markov models.
\end{itemize}

The paper is organized as follows. Section~\ref{sec:problem_setup} presents the problem formulation. Section~\ref{sec:main_results} contains the main results. Section~\ref{sec:experiments} presents numerical experiments illustrating the trade-offs and gains of the regret-optimal approach. Section~\ref{sec:proofs} provides the technical proofs. Section~\ref{sec:conclusions} concludes the paper. 

\section{Setting and Problem Formulation}\label{sec:problem_setup}
In this section, we formalize the control process as a disturbance-driven system. We then present our dynamic regret criterion that enables one to design causal policies that track benchmark policies, in our case, lookahead policies.

\subsection{Notation}
Variables are denoted by small letters, e.g., $s,w$ and subscripts denote their time indices, e.g., $s_t$. Alphabets of the corresponding variables are denoted by calligraphic letters, e.g., $s\in\mathcal S$. For integers $i\le j$, we write $w_{i}^{j}:=(w_i,w_{i+1},\dots,w_j)$. Bold letter refer to semi-infinite sequences e.g. $\mathbf{w}=(w_0,w_1,\dots)\in \W^{\mathbb{N}}$.

\subsection{Disturbance-Driven Dynamical Systems}
In stochastic MDPs, the state process typically evolves according to a transition kernel
$P_{S^+|S,A}(\cdot\mid s,a)$.
The transition kernel can be equivalently represented \cite{powell2009you} using the disturbance driven system
in \eqref{eq:dynamics}, 
where $W_t\sim P_W$ is a random variable, see the functional representation lemma, e.g., \cite{li2018strong}. The variable $w_t$ can be viewed as an exogenous disturbance since it does not depend on the state nor the action, and indeed it does not need to be defined as a random variable. In our setup, the exogenous disturbance $w_t$ is a deterministic quantity to be utilized for our regret objective, and we refer to \eqref{eq:dynamics} as a \emph{disturbance-driven} dynamical system. 

Let $\States$, $\A$, and $\W$ be finite sets of states, actions, and disturbances. The dynamic system starts from an initial state $s_0\in\States$ and evolves in discrete time according to \eqref{eq:dynamics}
where $a_t\in\A$ is the action and $w_t\in\W$ is a disturbance.
The system is controlled through a deterministic policy (a controller), defined as a sequence of decision rules at each time, i.e., $\pi=\{\mu_t\}_{t\ge 0}$ with 
\(
a_t=\mu_t(s_0, w_{0}^{t-1}).
\)
Note that $a_t$ is chosen before $w_t$ is revealed, but the controller at time $t$ can compute past and current states $s_1,\dots, s_t$ by knowing  $w_{0}^{t-1}$ and the initial state $s_0$. Indeed, the knowledge of $w_{0}^{t-1}$ implies that the action at time $t$ can utilize past actions and states, i.e., it is history-dependent. Yet, we will show in Theorem \ref{thm:infinite_main} that our optimal policy only depends on a simple function of the history.

The rewards are given by a time-invariant function 
\begin{align}
r:\States\times\A\times\W \to \mathbb{R},
\qquad |r(s,a,w)|\le R_{\max}<\infty .
\label{eq:reward_model}
\end{align}
For a trajectory $(s_t,a_t,w_t)_{t\ge 0}$, the reward is computed by 
\begin{align}\label{eq:rt_def}
r_t \;:=\; r(s_t,a_t,w_t), \qquad t=0,1,2,\dots.
\end{align}
The trajectory of the state-actions-disturbances triplet depends on the policy and the disturbance sequence, and thus we will use the shorthand $r_t$ only when these are clear from the context.  

For $\gamma\in(0,1)$, the discounted return of a policy $\pi$ from $s_0$ under a disturbance sequence
$\mathbf w$ is
\begin{align}
V(s_0,\pi;\mathbf w)
\;:=\;
\sum_{t=0}^{\infty}\gamma^t\, r_t
\label{eq:V_def}
\end{align}
Similarly, for a finite horizon $T\in\mathbb N$, the return is 
\[
V_T(s_0,\pi;w_{0}^{T-1})
\;:=\;
\sum_{t=0}^{T-1} r_t.
\]

\subsection{Lookahead policies and Dynamic Regret}
\label{subsec:policy_classes}
We first define the benchmark of lookahead policies. 
A history-dependent $k$-lookahead policy with $k\ge1$ is a sequence of deterministic functions $\pi^{k} := \{\mu_t^{k}\}_{t\ge 0}$ s.t.
\[
\mu_t^{k}: \States\times\W^{t+k} \to \A,
\quad
a_t=\mu_t^{k}(s_0, w_{0}^{t+k-1}).
\]
The class of $k$-lookahead policies is denoted by $\Pi_k$.
Note that $k=0$ recovers the class of causal policies, i.e., $\Pi_{\causal}:=\Pi_0$ since the controller has no lookahead. More importantly, lookahead policies are nested in the sense  
\(
\Pi_\causal = \Pi_{0}\subseteq \Pi_{1}\subseteq \Pi_{2}\subseteq \cdots
\).
We consider benchmarks in $\Pi_k$ that are described by the mapping $a_t^\benchmark=\beta(s_t^\benchmark,w_t^{t+k-1})$. The design of the benchmark needs to assume a model for the unknown future disturbances. Typical modeling of future disturbances are stochastic or worst-case\footnote{For discounted infinite-horizon settings, stationary rules using $(s_t^\benchmark,w_t^{t+k-1})$ suffice to maximize expected cumulative return under i.i.d.\ disturbances~\cite[Sec.~5.5 and Thm.~6.2.10]{puterman2014markov} or worst-case cumulative return under unrestricted disturbances~\cite[Sec.~3]{iyengar2005robust}.}.
For a causal controller $\pi_{\causal}\in\Pi_\causal$ and a benchmark policy $\pi_{\benchmark}$, we define
the regret as
\begin{align}
\mathrm{Regret}(s_0,\pi_{\causal},\pi_{\benchmark})
:=
\sup_{\mathbf w}
\left\{
V(s_0,\pi_{\benchmark};\mathbf w)
-V(s_0,\pi_{\causal};\mathbf w)
\right\}
\label{eq:regret_pipi}
\end{align}
The defining feature of \eqref{eq:regret_pipi} is that \emph{both returns are evaluated on the same
disturbance sequence $\mathbf w$}, ensuring an instance-wise comparison.
We can now define the optimal regret as a min-max problem over the path-wise regret in \eqref{eq:regret_pipi}
\begin{align}
\begin{aligned}
&\mathrm{Reg}_k^\star(s_0)
:=
\inf_{\pi_\causal\in\Pi_\causal}\ 
\mathrm{Regret}(s_0,\pi_\causal,\pi_\benchmark).
\end{aligned}
\label{eq:Reg_k_star}
\end{align}

The regret objective induces a disturbance-dependent weighting profile, when $V(s_0,\pi_\benchmark;\mathbf w)$ is low, the causal policy is not required to obtain a high return on that sequence, whereas when $V(s_0,\pi_\benchmark;\mathbf w)$ is high, the causal policy is compared against a more demanding benchmark. In the extreme case where $V(s_0,\pi_\benchmark;\mathbf w)$ is constant in $\mathbf w$, the regret problem reduces, up to an additive constant, to a robust-control formulation.

The definitions above adapt directly to a finite horizon $T\in\mathbb N$ by replacing \eqref{eq:regret_pipi} 
\(
\mathrm{Regret}_T(s_0,\pi_{\causal},\pi_{\benchmark})
:=
\sup_{w_{0}^{T-1}\in\W^T}
\left\{
V_T(s_0, \pi_\benchmark, w_0^{T-1})-V_T(s_0, \pi_\causal, w_0^{T-1})
\right\}
\)
\begin{align}
\mathrm{Reg}_{T,k}^\star(s_0) 
:=
\inf_{\pi_{\causal}\in\Pi_\causal}\ 
 \mathrm{Regret}_T(s_0,\pi_{\causal},\pi_{\benchmark})
\end{align}

\section{Main Results}
\label{sec:main_results}

The core challenge in the optimization of \eqref{eq:Reg_k_star} is that the causal controller and the $k$-lookahead benchmark act under different information patterns. 
We handle this mismatch by aligning the reward terms that become available at the same stage of the regret computation. At time \(t\ge k\), the most recent benchmark reward that can be paired with the causal trajectory is the delayed reward earned at time \(t-k\), while the causal contribution is the reward earned at time \(t\). Thus, after shifting the benchmark trajectory \(k\) stages behind the causal one, the relevant instantaneous comparison is between \(r^\benchmark_{t-k}\) and \(r^\causal_t\), corrected by the discount offset. This aligned comparison depends only on \(s_t^\causal, s_{t-k}^\benchmark\), and the last \(k\) disturbances. 
Since the aligned information has the same form at every stage \(t\ge k\), its optimal value is characterized by a stationary min-max Regret--Bellman operator.
The first \(k\) causal rewards have no preceding benchmark rewards to align with, and are therefore handled separately by a finite prefix DP whose terminal value is the continuation value reached after these \(k\) stages.
The DP state space is thus 
\begin{align}
x_t := \big(s^{\causal}_t, s^{\benchmark}_{t-k}, w_{t-k}^{t-1}\big)\ \in\ \X
:=\ \States\times\States\times \W^k,
\label{eq:X_chi_def}
\end{align}
the transition function is then
\begin{align}
\label{eq:Ft_def}
x_{t+1} 
&=
\Big(
f(s^{\causal}_t,a^\causal_t,w_t),
f\big(s^{\benchmark}_{t-k},a^\benchmark_{t-k},w_{t-k}\big),
w_{t-k+1}^{t-1}\#w_t
\Big) \nonumber\\
&:= F(x_t,a^\causal_t,w_t),
\end{align}
where $a^\benchmark_{t-k}=\beta(s_{t-k}^\benchmark, w_{t-k}^{t-1})$ and $w_{t-k+1}^{t-1}\#w_t:=(w_{t-k+1},\dots,w_{t-1}, w_t)$ denotes the appending of a vector with a scalar. The immediate cost is defined by a regret that aligns the information available to the controller
\begin{align}
\rho(x_t,a^{\causal}_t,w_t)
:=
r
\left(s^{\benchmark}_{t-k},
a^{\benchmark}_{t-k},
w_{t-k}\right)
-
\gamma^k r\left(s^{\causal}_t,
a^{\causal}_t,
w_t\right),
\label{eq:rho_def}
\end{align}
where $a^{\benchmark}_{t-k}=\beta(s_{t-k}^\benchmark, w_{t-k}^{t-1})$.
The factor \(\gamma^k\) accounts for the discount offset. 

The first DP can be summarized using a  \emph{Regret-Bellman operator} that computes a min-max DP transition.
\begin{definition}[Regret Bellman Operator]
\label{def:bellman_operator}
Let $\mathcal{B}(\X)$ be the
Banach space of bounded functions $J:\X\to\mathbb{R}$ with the sup norm.
The Regret-Bellman operator $\mathcal{R}:\mathcal{B}(\X)\to\mathcal{B}(\X)$ is given by
\begin{align}
(\mathcal{R}J)(x)
=
\min_{a^\causal\in\A}
\max_{w\in\W}
\Big\{
\rho(x,a^\causal,w)+\gamma\,J
\big(x^{+}\big)
\Big\},
\label{eq:bellman_discounted}
\end{align}
where $x^+:= F(x,a^\causal,w)$ as in~\eqref{eq:Ft_def}.
\end{definition}
We show below that the Regret-Bellman operator $\mathcal{R}$ is a contractive operator. Thus, it has a unique fixed-point that we denote by $J^\star_k(x)$, which corresponds to the regret-to-go from time $k$ onward. The fixed point can be computed using standard Value Iteration algorithm. The second DP utilizes $J^\star_k(x)$ as its terminal cost and accounts the first $k$ stages, summarized using the following \emph{Prefix DP}. 

\begin{definition}[Prefix DP]
\label{def:prefix}
For $0\le t < k$, define the prefix value functions as mappings $G_t:\States\times\W^t\to\mathbb{R}$. The value functions are computed by the backwards DP  
\begin{align}
G_t(s,w_{0}^{t-1}) 
&=
\min_{a^\causal\in\A}\ 
Q_t(s,w_0^{t-1},a^\causal),
\label{eq:Jt_def}
\\
Q_t(s,w_0^{t-1},a^\causal)
&:= 
\max_{w\in\W}
\left\{
\begin{aligned}
  G_{t+1}\big(f(s,a^\causal,w),w_{0}^{t-1}\#w\big) 
  -\gamma^t r(s, a^\causal, w)
\end{aligned}
\right\},\nonumber
\end{align}
with terminal condition $G_k(s,w_{0}^{k-1}) = J_k^\star(s, s_0, w_{0}^{k-1})$.
\end{definition}
We are ready to state our results regarding the optimal regret and the optimal policy. 
\begin{theorem}[Optimal regret and optimal policy]
\label{thm:infinite_main}
The optimal regret is given by
\begin{align}
\mathrm{Reg}_k^\star(s_0)=G_0(s_0),
\label{eq:Reg_equals_J0}
\end{align}
where $G_0(s_0)$ is computed by \eqref{eq:Jt_def} in Definition \ref{def:prefix}.

A regret-optimal policy for $t\ge k$ is given by
\begin{align}
a^\star(x)\in
\arg\min_{a^\causal\in\A}
\max_{w\in\W}
\Big\{\mspace{-3mu}
\rho(x,a^\causal,w) \mspace{-3mu}+\mspace{-3mu}\gamma J_k^\star\big(x^{+}\big)
\mspace{-3mu}\Big\},
\label{eq:areg_def}
\end{align}
where $\rho(\cdot)$ is the aligned regret in \eqref{eq:rho_def}, and $J_k^\star$ is the fixed-point of the Regret-Bellman operator.
For $0\le t<k$, a regret-optimal policy is given by the minimizer of \eqref{eq:Jt_def}, i.e., 
\begin{align*}
a_t^\star(s^\causal, w_{0}^{t-1})
& \in
\arg\min_{a^\causal\in\A}
Q_t(s^\causal,w_0^{t-1},a^\causal).
\end{align*}
\end{theorem}

Theorem \ref{thm:infinite_main} shows that the regret-optimal controller does not depend on the full history of disturbances. For $t\ge k$, its policy is stationary and $a_t$ only depends on $x_t=(s_t^\causal,s_{t-k}^\benchmark, w_{t-k}^{t-1})$, i.e. two states and a disturbance window of size $k$. Thus, the controller can be implemented with finite memory.
For comparison, a standard MDP design requires value iteration over a state space of size $|\States|$ 
while robust design introduces an additional maximization over disturbances, but retains the same state dimension. 
Our regret formulation is required to track both causal and lookahead states as well a length $k$ disturbance window, 
resulting in complexity scaling as $|\States|^2|\W|^k$.
In practice, however, our numerical examples demonstrate that small values of \(k\) may already yield policies that are competitive with respect to MDP and robust policies.

\subsection*{Optimal regret for the Finite-horizon case}
Fix a horizon $T\in\mathbb{N}$.
We retain the DP state space construction from
\eqref{eq:X_chi_def}-\eqref{eq:Ft_def} with $\gamma=1$.
The main structural difference from the infinite-horizon case is a boundary term corresponding to the last \(k\) benchmark rewards:
since each aligned benchmark reward is paired with a causal reward \(k\) stages later, these last benchmark rewards cannot be paired within the horizon and are collected
in a terminal tail value.
For a finite horizon, we allow time-dependent benchmark rules $a_t^\benchmark=\beta_t(s_t^\benchmark,w_t^{\min(t+k,T)-1})$.

\begin{definition}[Backward DPs]
\label{def:finite_backward_dp}
Define the following DPs:\\
\emph{Tail value:} define  $\Psi_k:\States\times\W^k\to\mathbb R$ by
\begin{align}
\Psi_k(s^{\benchmark}, w_{0}^{k-1})
:=
\sum_{\tau=T-k}^{T-1} 
r\big(s_\tau^{\benchmark},a_\tau^\benchmark,w_\tau\big),
\label{eq:tail_term}
\end{align}
with $s^{\benchmark}_{T-k} = s^\benchmark, w_{0}^{k-1}=w_{T-k}^{T-1}$, and $a_\tau^\benchmark=\beta_\tau(s_\tau^\benchmark,w_\tau^{T-1})$.\\
\emph{Regret DP:}
For $x=(s^{\causal},s^{\benchmark},w_0^{k-1})$,
and $k\le t<T$ let
\begin{align}
J_t(x)
=
\min_{a^\causal\in\A}\ 
\max_{w\in\W}
\Big\{
\rho(x,a^\causal,w) + J_{t+1}\big(x^{+}\big)
\Big\},
\label{eq:finite_bellman}
\end{align}
where $F$ and $\rho$ use $a^\benchmark=\beta_{t-k}(s^\benchmark,w_0^{k-1})$,
with terminal cost $J_T(x):=\Psi_k(s^{\benchmark}, w_{0}^{k-1})$.\\
\emph{Prefix DP:}
For $0 \le t < k$ define  $G_t:\States\times\W^t\to\mathbb R$ 
\begin{align}
G_t(s,w_{0}^{t-1})&=
\min_{a^\causal\in\A} 
Q_t(s,w_0^{t-1},a^\causal)
\label{eq:Jt_finite}
\end{align}
\(
Q_t(s,w_0^{t-1},a^\causal) 
\) as defined in \eqref{eq:Jt_def},
and the terminal condition $G_k(s,w_{0}^{k-1}):=J_k\big(s,s_0,w_{0}^{k-1}\big)$.
\end{definition}

\begin{theorem}[Finite-horizon regret DP]
\label{thm:finite_main}
The optimal regret is given by
\begin{align}
\mathrm{Reg}_{T,k}^\star(s_0)
= G_0(s_0).
\end{align}
For $k\le t<T$, an optimal regret-minimizing policy is any selector satisfying
\begin{align}
a_t^\star(x)\in
\arg\min_{a^\causal\in\A}\ 
\max_{w\in\W}
\left\{
\rho(x,a^\causal,w)
+
J_{t+1}\big(x^{+}\big)
\right\}.
\label{eq:finite_selector_regret}
\end{align}
For $0\le t<k$, an optimal selector satisfies
\begin{align*}
a_t^\star(s^\causal,w_{0}^{t-1})&\in
\arg\min_{a^\causal\in\A}
Q_t(s^\causal,w_0^{t-1},a^\causal).
\end{align*}
\end{theorem}
The optimal policy for the finite-horizon case is not stationary, but  it enables one to consider time-varying state evolution in~\eqref{eq:dynamics}.

\section{Numerical Examples}
\label{sec:experiments}
We study two examples to demonstrate the performance of our regret-optimal policy with respect to existing controllers. In both examples, we use the robust $k$-lookahead benchmark~\cite{iyengar2005robust}.

\textbf{Inventory management (IM): }
We consider the inventory management problem, e.g., \cite{zipkin2008old}, with dynamics 
\[
s_{t+1} = \min\left\{
|\States|, (s_{t} - w_t)^+ + a_t
\right\}
\]
and cost
\begin{align*}
c(s_t,a_t,w_t) = h\,(s_{t}-w_t)^+ + p\,(w_t-s_{t})^+ + o\cdot a_t,
\end{align*}
where $s_t$ is the inventory level, $a_t$ is the order and $w_t$ is demand, with
with \(\mathcal S=[1:40]\), \(\mathcal A=[1:12]\), and \(\mathcal W=[1:24]\).
We use the parameters in \cite{zipkin2008old}: the holding cost is $h=1$, the lost sales penalty is $p=9$, and the order unit cost is $o=2$. The discount factor is fixed to $\gamma=0.995$.
We take
\(
r(s_t,a_t,w_t) = -\,c(s_t,a_t,w_t),
\)
so maximizing reward is equivalent to minimizing cost.

We first consider a system driven by an i.i.d. disturbance sequence with a Poisson distribution of rate $\lambda$.
We compare regret-optimal policies with \(k=1,2,3\) against robust design, MDP policies designed for nominal rates \(\lambda\in\{2,5\}\), and the non-causal policy, which has access to future disturbances.

\begin{figure}[ht]
    \centering
    \includegraphics[
        width=0.99\columnwidth,
        trim=0pt 0pt 8pt 0pt,
    ]{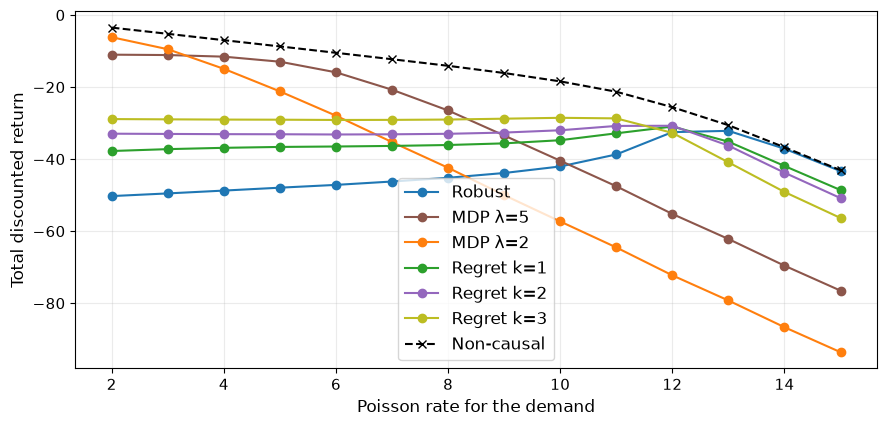}
    \caption{Performance of controllers in inventory management as a function of an i.i.d. Poisson demand rate \(\lambda\).}
    \label{fig:InfinitePossionChangeVsCost}
\end{figure}

Fig.~\ref{fig:InfinitePossionChangeVsCost} shows the cumulative rewards as a function of the true demand rate. The MDP policies perform well near their nominal rates but degrade under misspecification, while the robust policy is conservative at low rates and becomes competitive only in high-demand regimes, where it nearly overlaps with the non-causal benchmark. The regret-optimal policies interpolate between them and remain approximately uniform over a wide range of rates. At low \(\lambda\), additional preview is useful to the benchmark, and the regret-optimal controller benefits from tracking it. At high \(\lambda\), however, preview saturates, since the locally optimal action depends less on future demand; consequently, larger \(k\) no longer provides an advantage, and the less conservative \(k=1\) controller can perform better. This shows that small lookahead depths already capture much of the benefit of the regret formulation.

We also consider a non-i.i.d.\ demand model generated by an underlying two-state Hidden Markov Model (HMM). 
The hidden Markov state indicates whether the system is in a low-demand Poisson rate denoted by $\lambda_L$, or high-demand rate chosen as $\lambda_L + 3$. 
The Markov chain evolves symmetrically with persistence probability \(0.9\). This introduces temporal dependence in the disturbance sequence, so demands are no longer sampled independently across time.

Table \ref{table:policy_comparison}
presents the discounted reward of several values of $\lambda_L$, regret-optimal controllers balance the MDP and robust policies, and in an intermediate-demand regime the regret-optimal controllers can even outperform both.

\begin{table}[h]
\caption{IM - Performance for HMM demands}
\label{table:policy_comparison}
\begin{center}
\begin{tabular}{|c||c|c|c|}
\hline
 & $\lambda_{L}=4$ & $\lambda_{L}=8$ & $\lambda_{L}=16$ \\
\hline
Robust policy & -57.5 & -51.4 & \textbf{-72.2} \\
\hline
MDP with $\lambda=2$ & -29.8 & -65.2 & -133.7 \\
\hline
MDP with $\lambda=5$ & \textbf{-19.0} & -45.2 & -113.0 \\
\hline
Regret $k=1$ & -44.3 & -41.9 & -78.9 \\
\hline
Regret $k=2$ & -39.9 & \textbf{-38.5} & -82.0 \\
\hline
\end{tabular}
\end{center}
\end{table}

\textbf{Packet routing: }
We consider a two-link packet routing problem inspired by stochastic network control.
At each step,
a source receives $\alpha \in \{\alpha_{\min},\dots,\alpha_{\max}\}$ packets,
and the controller can dispatch up to $D$ packets over two parallel links.
Link $A$ is congestion-prone with per-step capacity $c_A \in [0:C_A]$; 
link $B$ is reliable with constant capacity~$c_B$. 
The disturbance is the pair $w = (c_A, \alpha)$.
A state $s = (n_s, q_A, q_B)$ records packets at the source and on each link queue, with max capacities $N_s, Q_A, Q_B$. 
An action $a = (a_A, a_B)$ is the dispatches on the two links and satisfies 
$a_A + a_B \le \min(D, n_s)$. 
The dynamics are 
\(
n_s^+ = \min\bigl(N_s,\ n_s - a_A - a_B + \alpha\bigr),
\)
\(
q_i^+ = \max\bigl(0,\ \min(Q_i, q_i + a_i) - c_i\bigr), i \in \{A, B\}.
\)
The reward is
\(
r_t = R_{\text{thr}}\,\delta_t - h_s n_s^+ - h_A q_A^+ - h_B q_B^+ - P_{\text{ovf}}\,o_t,
\)
where $\delta_t = \sum_{i\in\{A,B\}} \min(\min(Q_i, q_i + a_i),\, c_i)$ is the total delivered throughput 
and $o_t = \sum_{i\in\{A,B\}} (q_i + a_i - Q_i)^+$ counts overflowing packets.
We use $R_{\text{thr}}=2$, $(h_s,h_A,h_B)=(0.1,0.5,0.3)$, $P_{\text{ovf}}=10$.
\(\alpha_t\in \{2,5\}\), and \(C_A=9\).

We assume an optimistic nominal $w_t\sim p_{\text{nom}}$ (high-capacity, low-arrival).
Controllers are evaluated under distributions $p_\eta$, an exponential tilt of $p_{\text{nom}}$, i.e.
\(
p_\eta(w) \;\propto\; p_{\text{nom}}(w)\, \exp\bigl(\eta\, b(w)\bigr),
\)
with $b(w) = (C_A - c_A) + 2\bigl[\alpha - c_B - \min(c_A,\, D-c_B)\bigr]_+$.
The first term penalizes low link-$A$ capacity; the second penalizes arrivals that exceed the step's service capability. 
At $\eta=0$ this recovers $p_{\text{nom}}$,
As $\eta$ grows, mass shifts toward joint disturbances that are both capacity-limited and overloaded.

\begin{figure}[ht]
    \centering
    \includegraphics[width=0.99\linewidth]{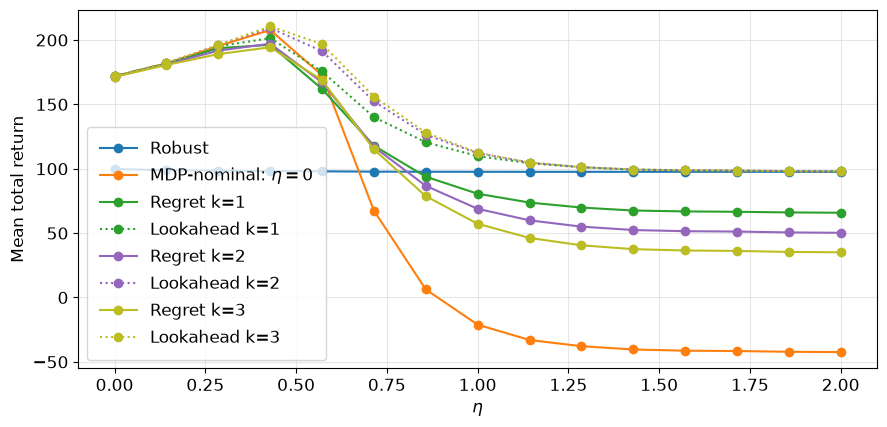}
    \caption{
    Performance comparison under varying \(\eta\) values.
    }
    \label{fig:packetRoutingEtaSweep}
\end{figure}

Fig.~\ref{fig:packetRoutingEtaSweep} shows the sensitivity of the controllers to the disturbance tilt. The nominal MDP performs best near $\eta=0$ but degrades sharply as $\eta$ grows; the robust policy is largely insensitive to $\eta$ but conservative throughout. The regret-optimal policies interpolate between the two already for small lookahead depths, while in the more adverse regimes, smaller values of \(k\) provide additional gains.

The $k$-lookahead curves are the robust \(k\)-lookahead benchmarks.
Hence, they are optimized for worst-case rather than expected return.

\section{Proofs}
\label{sec:proofs}

Throughout this section, we use the shorthands 
\(
r_t^\causal := r(s_t^\causal, a_t^\causal, w_t),
\quad
r_t^\benchmark := r(s_t^\benchmark, a_t^\benchmark, w_t),
\quad\rho_t:=\rho\big(x_t,a_t^\causal,w_t\big),
\)
with $\rho(\cdot)$ as defined \eqref{eq:rho_def}.

\newcommand{\pstep}[1]{\par\smallskip\noindent\textbf{#1}\quad}

\begin{proof}[Proof of Theorem~\ref{thm:infinite_main}]
The proof has three steps. We begin by writing the optimal regret as a nested two-stage optimization. We then show that the inner optimization corresponds to an infinite-horizon DP with discounted rewards, and its optimal value is equal to the fixed-point of the Regret-Bellman operator in \eqref{eq:bellman_discounted}. Finally, we use the fixed-point to compute the outer optimization of the regret.

\pstep{Step 1: Regret as a nested optimization.}
In this step, we show that  
$\mathrm{Reg}^\star(s_0)$ can be written as the optimization
\begin{align}\label{eq:regret_nested_opt}
\min_{\{\mu_i^C\}_{i=0}^{k-1}} 
\max_{w_0^{k-1}}
\Bigg[ 
-
\sum_{t=0}^{k-1}\gamma^t r_t^\causal
+
\min_{\{\mu_i^C\}_{i=k}^{\infty}} 
\max_{w_{k}^{\infty}\in\W^{\mathbb N}}
\sum_{t=0}^{\infty}\gamma^t\,
\rho_{t+k}
\Bigg].
\end{align}
The regret, for fixed policies and disturbance, is defined as the discounted returns' difference. We can then write
\begin{align}\label{eq:proof_decomp_iden}
\nonumber
&V(s_0, \pi_\benchmark, \mathbf{w}) 
- V(s_0, \pi_\causal, \mathbf{w}) 
= \sum_{t=0}^\infty \gamma^t (r_t^L - r^C_t)\\[-0.8em]
&= -\sum_{t=0}^{k-1}\gamma^t\,r_t^\causal + \sum_{t=0}^\infty \gamma^t\underbrace{(r_t^\benchmark - \gamma^k r^\causal_{t+k})}_{\rho_{k+t}},
\end{align}
To show \eqref{eq:regret_nested_opt},
We split $\mathbf{w}=(w_{0}^{k-1},w_{k}^{\infty})$ and $\pi_\causal=\{\mu_t^C\}_{t\ge0}=(\{\mu_t^C\}_{t=0}^{k-1},\{\mu_t^C\}_{t\ge k})$.
The rewards $r_t^\causal$ for $t<k$ only depend on 
$\{\mu_t^C\}_{t=0}^{k-1}$
and $w_0^{k-1}$, and are independent of $\{\mu_t^C\}_{t\ge k}$ and $w_{k}^{\infty}$.
Since the causal policy may depend on the observed prefix, its continuation can be optimized separately for each $w_0^{k-1}$. Combining with \eqref{eq:proof_decomp_iden} yields \eqref{eq:regret_nested_opt}.
Now fix a prefix strategy $\{\mu_i^C\}_{i=0}^{k-1}$ and a prefix disturbance $w_0^{k-1}$.
By determinism of \eqref{eq:dynamics}, at time $k$ the causal state $s_k^\causal$ is uniquely determined.
Moreover, once the prefix disturbance sequence \(w_0^{k-1}\) is fixed, the causal state \(s_k^\causal\) is determined by the prefix causal actions and disturbances. The benchmark trajectory, when viewed \(k\) stages behind the causal one, starts from \(s_0\) and carries the same disturbance buffer \(w_0^{k-1}\). Hence the continuation problem depends on the past only through the information state
\(
x_k \ :=\ \big(s_k^\causal,\ s_0,\ w_0^{k-1}\big).
\)\\
Before we proceed, we denote the inner optimization as 
\begin{align}
C(x)
:=
\min_{\{\mu_i^C\}_{i=k}^{\infty}}
\max_{w_{k}^{\infty}\in\W^\mathbb{N}}
\left[
\sum_{n=k}^{\infty}\gamma^{n-k}
\rho(x_n, a_n^\causal, w_{n})
\right].
\label{eq:continuation_def}
\end{align}
Thus, for $x_k=(s_k^\causal,s_0, w_{0}^{k-1})$,
the optimal regret is 
\begin{align}
\mathrm{Reg}^\star(s_0) &= \min_{\{\mu_i^C\}_{i=0}^{k-1}} \max_{w_0^{k-1}} \left\{ -\sum_{t=0}^{k-1}\gamma^t r_t^\causal
+ C(x_k) \right\}.
\label{eq:prefix_plus_continuation}
\end{align}

\pstep{Step 2: Computing $C(x)$.}
We prove that $C(x)$ is equal to the fixed-point of the Regret-Bellman operator (Def.~\ref{def:bellman_operator}).
Let \(J,J'\in\mathcal B(\X)\). 
For any fixed 
\((x,a^\causal,w)\),
we have 
\begin{align*}
\Big|
\rho(\cdot)
+\gamma J\big(x^+\big)
-\rho(\cdot)
-\gamma J'\big(x^+\big)
\Big|
\le \gamma\|J-J'\|_\infty.
\end{align*}
Taking \(\max_w\) and taking
\(\min_{a^\causal}\) preserve the bound. 
Therefore
\(
\|\mathcal RJ-\mathcal RJ'\|_\infty
\le
\gamma\|J-J'\|_\infty.
\)
Hence \(\mathcal R\) is a \(\gamma\)-contraction on \(\mathcal B(\X)\), and by Banach's
fixed-point theorem it admits a unique fixed point denoted \(J_k^\star\).

Note that $s^\benchmark_{t-k}$ and $a^\benchmark_{t-k}=\beta(s^\benchmark_{t-k},w_{t-k}^{t-1})$ can be computed using the causal history for $t\ge k$.
The continuation value \(C(x)\) defined in
\eqref{eq:continuation_def} depends on the past through the state \(x\), 
by a standard one-step decomposition of \eqref{eq:continuation_def},
\(C = \mathcal R C\). 
By uniqueness of the fixed point,
\(C(x)=J_k^\star(x)\) for all \(x\in\X\).

Because \(\A\) is finite, the minimum in \eqref{eq:bellman_discounted} is attained at
each \(x\in\X\). Let \(a^\star:\X\to\A\) be any selector satisfying
\eqref{eq:areg_def}. Since \(J_k^\star\) is the fixed point of \(\mathcal R\), we have \(\forall x\in\X\),
\[
J_k^\star(x)
=
\max_{w\in\W}
\Big\{
\rho(x,a^\star(x),w)+\gamma J_k^\star(x^+)
\Big\}.
\]
Hence \(a^\star\) attains \(J_k^\star\), and is
optimal for all \(t\ge k\).

\pstep{Step 3: Reduction to the prefix DP.}
Substituting \(C(x_k)=J_k^\star(x_k)\) in
\eqref{eq:prefix_plus_continuation} reduces the outer optimization to the
\(k\)-stage prefix DP \eqref{eq:Jt_def} with terminal condition
\(G_k(s,w_0^{k-1})=J_k^\star\bigl((s,s_0,w_0^{k-1})\bigr)\). Evaluating at
\(t=0\) gives \(\mathrm{Reg}_k^\star(s_0)=G_0(s_0)\), proving
\eqref{eq:Reg_equals_J0}. The prefix-DP selector is optimal for \(t<k\).
\end{proof}

\begin{proof}[Proof of Theorem~\ref{thm:finite_main}]
For any $\pi_\causal\in\Pi_{\causal}$, $\pi_\benchmark\in\Pi_k$, and $w_{0}^{T-1}\in\W^T$, the regret admits the decomposition
\begin{align}
\label{eq:finite_gap_decomp}
&V_{T}(\cdot, \pi_\benchmark;\cdot) 
- V_{T}(\cdot, \pi_\causal;\cdot)
 =
 \sum_{t=k}^{T-1} \rho_t 
+ \sum_{i=T-k}^{T-1}r_{i}^\benchmark
- \sum_{j=0}^{k-1}r_j^\causal.
\end{align}
The tail value is computed by evaluating the benchmark's return for $s_{T-k}^\benchmark,w_{T-k}^{T-1}$.

Fix $k\le t\le T$. For \(x\in\X\), let \(\widehat J_t(x)\) denote the value of the
stage-\(t\) continuation game starting from \(x_t=x\), namely the minmax problem over
causal actions from stages \(t,\dots,T-1\) and disturbances from stages \(t,\dots,T-1\), with running cost
\(\rho\) and terminal tail value \(\Psi_k\). 
At stage $t$, the benchmark action is determined by $\beta_{t-k}$.
By one-step decomposition of this finite
game, \(\widehat J_t\) satisfies the backward DP \eqref{eq:finite_bellman} with
terminal condition \(J_T=\Psi_k\). Hence, by backward induction,
\(
\widehat J_t(x)=J_t(x), \quad k\le t\le T.
\)

Now fix a causal policy prefix up to stage \(k-1\) and a disturbance prefix
\(w_{0}^{k-1}\). By the decomposition \eqref{eq:finite_gap_decomp}, the total regret
splits into the prefix loss \(-\sum_{t=0}^{k-1} r_t^\causal\) and a continuation term
that depends on the past only through
\(
x_k=(s_k^\causal,s_0,w_{0}^{k-1})\in\X,
\)
and is equal to
\(J_k(x_k)\). Therefore
\begin{align}
&\mathrm{Reg}_{T,k}^\star(s_0) 
=
\inf_{\pi_\causal\in\Pi_\causal}\ 
\sup_{w_{0}^{k-1}\in\W^k}
\left[
J_k\big(x_k\big)
-\sum_{t=0}^{k-1} r_t^\causal
\right].
\label{eq:finite_reduction_to_prefix}
\end{align}

Define \(\widehat G_t(s,w_{0}^{t-1})\) as the value of the resulting \(k\)-stage prefix
game from stage \(t\), with terminal value \(J_k((s,s_0,w_{0}^{k-1}))\). Again by
one-step decomposition, \(\widehat G_t\) satisfies the prefix DP in \eqref{eq:Jt_finite}.
Hence, by backward induction, $\widehat G_t=G_t$ for $0\le t\le k$.
Evaluating \eqref{eq:finite_reduction_to_prefix} at \(t=0\) gives
\[
\mathrm{Reg}_{T,k}^\star(s_0)=G_0(s_0).
\qedhere
\]
\end{proof}

\section{Conclusions} \label{sec:conclusions}
We studied worst-case dynamic regret in finite-state disturbance-driven systems, where a causal controller is evaluated against a $k$-lookahead benchmark. For the discounted infinite-horizon setting, we derived a Regret–Bellman operator on a tracking state space $\X$ and showed that it is a $\gamma$-contraction, ensuring existence and uniqueness of a fixed point and enabling value iteration. Coupling this fixed point with a finite-prefix DP yields the optimal regret value and a regret-optimal controller that is stationary for all $t \ge k$. Similar solution structure was also shown in the finite-horizon setting.

Numerical experiments illustrate the qualitative behavior of regret-optimal policies, 
by tracking a lookahead benchmark in a worst-case regret sense, 
they remain robust to disturbance misspecification, interpolate between MDP-based and robust designs as problem parameters vary.

Future work includes partially observable systems and the distributionally robust formulation in which a nominal probabilistic model for disturbance is available, but the true distribution is chosen from an ambiguity set around the nominal distribution \cite{al2023distributionally, kargin2024wasserstein}.

\bibliographystyle{IEEEtran}
\bibliography{references}

\end{document}